\documentclass[a4paper,12pt]{article}
\usepackage{hyperref}
\usepackage{xcolor}
\usepackage{amssymb,amsthm,amsmath} 
\usepackage[a4paper]{geometry}
\usepackage{datetime2}
\usepackage[utf8]{inputenc}
\usepackage[italian,english]{babel}
\usepackage{mathrsfs}
\usepackage{bbm}
\usepackage{enumitem}

\numberwithin{equation}{section}

\newcommand{\R}{\mathbb{R}}
\newcommand{\Z}{\mathbb{Z}}
\newcommand{\s}{\mathbb{S}}

\newcommand{\D}{\mathcal{D}}

\newcommand{\N}{\mathbb{N}}
\newcommand{\ep}{\varepsilon}

\newcommand{\Fgpl}{F_{\gamma,p,\lambda}}
\newcommand{\Egpl}{E_{\gamma,p,\lambda}}
\newcommand{\psigpl}{\psi_{\gamma,p,\lambda}}
\newcommand{\DFgpl}{DF_{\gamma,p,\lambda}}

\newcommand{\Vis}{\mathcal{V}}
\newcommand{\Imp}{\mathcal{I}}

\newcommand{\diam}{\operatorname{diam}}
\newcommand{\loc}{\mathrm{loc}}

\newcommand{\Leb}{\ensuremath{\mathscr{L}}}
\DeclareMathOperator{\osc}{osc}

\newtheorem{thm}{Theorem}[section]

\newtheorem{rmk}[thm]{Remark}
\newtheorem{prop}[thm]{Proposition}

\newtheorem{cor}[thm]{Corollary}

\newtheorem{lemma}[thm]{Lemma}
\newtheorem{open}{Open problem}

\title{Lower bounds and Poincaré inequality for the Brezis-Seeger-Van Schaftingen-Yung functionals}

\author{Nicola Picenni\vspace{1ex}\\ 
{\normalsize Università degli Studi di Roma Tor Vergata} \\
{\normalsize Dipartimento di Matematica}\\ 
{\normalsize Via della Ricerca Scientifica, 1 - 00133 Roma (Italy)}\\  
{\normalsize e-mail: \texttt{picenni@mat.uniroma2.it}}
}
\date{}

\begin{document}

\maketitle

\begin{abstract}

We consider a family of non-local and non-convex functionals depending on three real parameters $(\gamma,p,\lambda)$, which was introduced by H.~Brezis, A.~Seeger, J.~Van~Schaftingen and P.~L.~Yung. In the case $\gamma<-p$ we prove a lower bound for the Gamma-liminf of these functionals as $\lambda\to 0^+$ and a Poincaré-type inequality. We also show that such inequality cannot hold if $\gamma>0$, and we establish a modified version in this case.

These results answer some open questions posed by these authors and by H.~M.~Nguyen.

\vspace{6ex}

\noindent{\bf Mathematics Subject Classification (MSC2020):} 49J45, 26D10, 46E35, 26B30.

\vspace{6ex}



\end{abstract}

\thispagestyle{empty}
\setcounter{page}{0}

\clearpage

\section{Introduction}

In this paper we address some questions concerning a class of non-local, non-convex functionals introduced by H.~Brezis, A.~Seeger, J.~Van Schaftingen and P.-L.~Yung in~\cite{2022-BSVY-Lincei,BSVY}, which is related to the Sobolev norm and the total variation.

In order to introduce these functionals, we consider a positive integer $N$, an open set $\Omega \subseteq \R^N$, and three real parameters $\gamma\in\R$, $p\geq 1$ and $\lambda>0$.

Then we define 
\begin{equation}\label{defn:Fgpl}
\Fgpl(u,\Omega):=\lambda^p \nu_\gamma(\Egpl(u,\Omega))
\qquad
\forall u\in L^1_\loc(\Omega),
\end{equation}
where $\nu_\gamma$ is the measure on $\R^N\times \R^N$ defined by
\begin{equation}\label{defn:nu_gamma}
\nu_\gamma (A):=\iint_A |y-x|^{\gamma-N}\,dx\,dy,
\end{equation}
and the set $\Egpl(u,\Omega)$ is the set defined as follows
\begin{equation}\label{defn:Egpl}
\Egpl(u,\Omega):=\left\{(x,y)\in \Omega\times\Omega:|u(y)-u(x)|>\lambda |y-x|^{1+\gamma/p}\right\}.
\end{equation}

Prior to \cite{2022-BSVY-Lincei,BSVY}, the functionals (\ref{defn:Fgpl}) had already been considered in the special case $\gamma=-p$ (see \cite{2006-Nguyen-JFA,2006-CRAS-BN,2011-Nguyen-Duke,2011-CalcVar-Nguyen}), and in the special case $\gamma=N$ (see \cite{2021-BVY-PNAS,2021-BVY-CalcVar,2022-Poliakovsky-JFA}). Note that these cases are exactly those for which either the exponent in (\ref{defn:Egpl}) or the one in (\ref{defn:nu_gamma}) vanishes.

In order to state describe the known results involving the functionals (\ref{defn:Fgpl}), let us introduce some convenient notation. Let $\dot{W}^{1,p}(\Omega)$ denote the space of functions $u\in L^1 _{\loc}(\Omega)$ whose distributional gradient $\nabla u$ belongs to $L^p(\Omega)$, and let $\dot{BV}(\Omega)$ denote the space of functions $u\in L^1 _{\loc}(\Omega)$ whose distributional gradient $Du$ is a finite $\R^N$-valued Radon measure on $\Omega$. Then for every $u\in L^1 _{\loc}(\Omega)$ we set
\begin{equation}
F_p(u,\Omega):=\begin{cases}
\|\nabla u\|_{L^p(\Omega)} ^p & \mbox{if $p>1$ and $u\in \dot{W}^{1,p}(\Omega)$},
\\[0.5ex]
|Du|(\Omega) & \text{if $p=1$ and $u\in \dot{BV}(\Omega)$},
\\
+\infty & \text{otherwise}.
\end{cases}
\label{defn:Fp}
\end{equation}

Let us define also the geometric constant 
\begin{equation}\label{defn:C_N}
C_{N,p}:=\int_{\s^{N-1}} |\langle v,x\rangle|^p\,d\mathcal{H}^{N-1}(x),
\end{equation}
where $v\in \R^N$ is any unit vector.

\subsection{Previous literature}

Let us summarize some of the main results concerning (\ref{defn:Fgpl}) that are of interest for us. We divide the cases $\gamma>0$ and $\gamma<0$, since they present some differences, and we discuss separately the case $\gamma=-p$, in which more precise results are available.

We also remark that the case $\gamma=0$ is quite pathological and presents no relevant limits, so we exclude it in the sequel and we refer to \cite[Section~5]{BSVY} for more details.

Moreover, for simplicity, we state the convergence results only in the case $\Omega=\R^N$, even though most of them have been proved also in local versions.

Before listing the results, let us remark that the relevant limit is the one for $\lambda\to +\infty$ when $\gamma>0$ and for $\lambda\to 0^+$ when $\gamma<0$ (see the remark after Theorem~1.1 in \cite{BSVY}).

\paragraph{The case $\gamma=-p$.}

This case was the first one to be considered and it is now well understood. We recall just a few clean results, but many generalizations are available. We mention some of them at the end of the present subsection.

\begin{itemize}
    \item (Pointwise convergence \cite{2006-Nguyen-JFA}) For every $p>1$ it turns out that
    $$\lim_{\lambda\to 0^+} F_{-p,p,\lambda}(u,\R^N)=
    \frac{C_{N,p}}{p}\cdot F_p(u,\R^N) 
    \qquad 
    \forall u \in L^p(\R^N).$$

    In the case $p=1$, the analogous limit holds true when $u\in C^1_c(\R^N)$, but there exists a function $u\in W^{1,1}(\R^N)$ for which
    $$\lim_{\lambda\to 0^+} F_{-1,1,\lambda}(u,\R^N)= +\infty.$$

    \item (Gamma-convergence \cite{2011-Nguyen-Duke,2020-APDE-AGMP}) For every $p\geq 1$ it turns out that
    $$\Gamma\mbox{--}\lim_{\lambda\to 0^+} F_{-p,p,\lambda}(u,\R^N)=
    \frac{C_{N,p}}{p}\cdot k_{p}\cdot F_p(u,\R^N) 
    \qquad \forall u \in L^p(\R^N),$$
    where
    \begin{equation*}
    k_p:=\begin{cases}
    \dfrac{1}{p-1}\left(1-\dfrac{1}{2^{p-1}}\right) &\text{if $p>1$},
    \\[1.5ex]
    \log 2 &\text{if $p=1$}.
    \end{cases}
\end{equation*}
    Note that the Gamma-limit is strictly smaller than the pointwise limit.

    \item (Poincaré inequality \cite{2011-CalcVar-Nguyen}) Let us assume that $\Omega$ is a ball. Then for every $p\geq 1$ there exists a constant $K_{N,p}>0$, depending only on $N$ and $p$, such that
    \begin{equation}\label{nguyen-poincaré}
    \iint_{\Omega\times \Omega} |u(y)-u(x)|^p\,dxdy \leq K_{N,p}\left(|\Omega|^{\frac{N+p}{N}} F_{-p,p,\lambda}(u,\Omega) + \lambda ^p |\Omega|^2 \right),
    \end{equation}
    for every $u\in L^p(\Omega)$ and every $\lambda>0$.
\end{itemize}

\paragraph{Previous literature in the case $\gamma>0$.}

More recently, the case $\gamma=N$ (so that $\nu_N$ coincides with the Lebesgue measure) was considered in \cite{2021-BVY-PNAS,2021-BVY-CalcVar,2022-Poliakovsky-JFA} (see also \cite[Section~9]{Brezis-Lincei}). Shortly after the publication of these works, the general case $\gamma\in \R$ was introduced in \cite{2022-BSVY-Lincei,BSVY}. Since, to our knowledge, by now there are not results that are specific to the case $\gamma=N$, we consider directly the more general case $\gamma>0$.

\begin{itemize}
\item (Supremum estimate from above and below \cite{BSVY}) For every $p\geq 1$ and $\gamma>0$ there exist two positive constants $c_1(N,\gamma,p)$ and $c_2(N,\gamma,p)$ such that
\begin{equation*}
c_1(N,\gamma,p)F_p(u,\R^N)\leq \sup_{\lambda>0} \Fgpl(u,\R^N)\leq c_2(N,\gamma,p)F_p(u,\R^N) \quad \forall u\in L^1 _{\loc}(\R^N).
\end{equation*}

\item  (Pointwise convergence for $p>1$ \cite{BSVY,GP-gammaliminf}) For every $p> 1$ and $\gamma>0$ it holds that
\begin{equation}\label{lim_W1p}
\lim_{\lambda\to +\infty} \Fgpl(u,\R^N)=
C_{N,p}\cdot\frac{1}{\gamma}\cdot F_p(u,\R^N) 
\qquad 
\forall u\in L^1 _{\loc}(\R^N).
\end{equation}

\item  (Pointwise convergence for $p=1$) In the case $p=1$, the pointwise convergence is more delicate. Combining results from \cite{2023-Picenni,2023-Lahti,GP-gammaliminf} we now know that
\begin{equation}\label{liminf_BV_2}
\liminf_{\lambda\to +\infty} F_{\gamma,1,\lambda}(u,\R^N) \geq C_{N,1}\left\{\frac{1}{\gamma} |D^a u|(\R^N) + \frac{1}{\gamma+1} \left(|D^j u|(\R^N) + |D^c u|(\R^N)\right)\right\},
\end{equation}
for every $\gamma>0$ and $u\in L^1_{\loc}(\R^N)$, where $D^a u$, $D^j u$ and $D^c u$ denote, respectively, the absolutely continuous part, the jump part and the Cantor part of the distributional derivative of $u$. Moreover, it holds that
\begin{equation*}
\lim_{\lambda\to +\infty} F_{\gamma,1,\lambda}(u,\R^N) = 
C_{N,1}\left\{\frac{1}{\gamma} |D^a u|(\R^N) + \frac{1}{\gamma+1} |D^j u|(\R^N)\right\}
\end{equation*}
for every $u\in \dot{SBV}(\R^N)$, namely for $u\in \dot{BV}(\R^N)$ with $D^c u=0$. However, the pointwise limit might not exist for a general $u\in \dot{BV}(\R^N)$, as it was shown in \cite[Example~3.22]{2023-Lahti}.

\item (Gamma-convergence \cite{GP-gammaliminf,arxiv-gammalim})
For every $p\geq 1$ and $\gamma>0$ there exists a constant
\begin{equation}\label{gp_est}
k(N,\gamma,p)\geq C_{N,p}\cdot \frac{\log 2} {2^{1+\gamma}-1}
\end{equation}
such that
\begin{equation}\label{gammalim-implicit}
\Gamma-\lim_{\lambda\to +\infty} \Fgpl(u,\R^N)= k(N,\gamma,p) \cdot F_p(u,\R^N) \qquad \forall u\in L^1_{\loc}(\R^N).
\end{equation}

Unfortunately, the constant $k(N,\gamma,p)$ is defined in \cite{arxiv-gammalim} in an implicit way as the solution of a cell problem, so we still do not know if the Gamma-limit coincides with the pointwise limit (or the relaxation of the pointwise liminf when $p=1$), or if a gap phenomenon occurs as in the case $\gamma=-p$. In fact, the estimate (\ref{gp_est}), which comes from \cite{GP-gammaliminf}, is probably not sharp, but the proof of (\ref{gammalim-implicit}) in \cite{arxiv-gammalim} gives no information about this question.

\item (Gamma-convergence with respect to the area-strict convergence \cite{Lahti-Li}) If $p=1$, then for every $\gamma>0$ it holds that
\begin{multline*}
\Gamma_{AS}-\lim_{\lambda\to +\infty} F_{\gamma,1,\lambda}(u,\R^N) \\
= C_{N,1}\left\{\frac{1}{\gamma} |D^a u|(\R^N) + \frac{1}{\gamma+1} \left(|D^j u|(\R^N) + |D^c u|(\R^N)\right)\right\},
\end{multline*}
where the Gamma-limit in the left-hand side is intended with respect to the area-strict convergence, instead of the usual $L^1 _{\loc}$ topology. We refer to \cite{Lahti-Li} for the definition of area-strict convergence and more details on this statement.

\end{itemize}

\paragraph{Previous literature in the case $\gamma<0$.}

As opposite to the case $\gamma>0$, much less is known in the case $\gamma<0$, outside of the special case $\gamma=-p$. This might be due to the fact that the case $\gamma=N$ appeared in Brezis' list of open problems \cite{Brezis-Lincei}, which attracted the attention of many researchers, but also to the fact that the negative case exhibits some additional complications. In fact, there are some additional relevant thresholds at $\gamma=-p$, due to the change of sign in the exponent appearing in (\ref{defn:Egpl}), and at $\gamma=-1$, which separates the regimes for which jump-type discontinuities cost infinite energy or not.

Anyway, let us summarize the state of the art in this setting.

\begin{itemize}
    \item (Supremum estimate from above and below \cite{BSVY}) For every $p>1$ and $\gamma<0$ there exist two positive constants $c_1(N,\gamma,p)$ and $c_2(N,\gamma,p)$ such that
\begin{equation*}
c_1(N,\gamma,p)F_p(u,\R^N)\leq \sup_{\lambda>0} \Fgpl(u,\R^N)\leq c_2(N,\gamma,p) F_p(u,\R^N)\quad \forall u\in L^1 _{\loc}(\R^N).
\end{equation*}

If $p=1$ the statement above holds if and only if $\gamma<-1$.

    \item (Lower bound when $\gamma\in [-p,-1]$ \cite{Nguyen-survey}) For every $p \geq 1$ and $\gamma\in [-p,-1]$ there exists a constant $c(N,p)>0$, depending only on $N$ and $p$, such that
    \begin{equation}\label{liminf_est_nguyen}
    \liminf_{\lambda\to 0^+} \Fgpl(u,\R^N) \geq c(N,p)\cdot F_p(u,\R^N)\qquad \forall u\in L^1 _{\loc}(\R^N).
    \end{equation}

    \item (Pointwise convergence on Sobolev spaces \cite{BSVY}) For every $p>1$ and $\gamma<0$ it holds that
    \begin{equation}\label{lim-sobolev}
    \lim_{\lambda\to 0^+} \Fgpl(u,\R^N)=
C_{N,p}\cdot\frac{1}{|\gamma|}\cdot F_p(u,\R^N)  \qquad \forall u\in \dot{W}^{1,p} (\R^N).
    \end{equation}

If $p=1$ the statement above holds only if $\gamma<-1$, or if $u\in C^1 _c(\R^N)$, but it fails generically for $\gamma\in [-1,0)$ and $u\in \dot{W}^{1,1}(\R^N)$ (see \cite[Theorem~1.8]{BSVY} for precise statements about the counterexamples). Moreover, in the case $p=1$, there are not known formulas nor estimates for the $\dot{BV}$ and $\dot{SBV}$ cases.

We point out that for $\gamma\in [-p,-1]$ the pointwise limit formula (\ref{lim-sobolev}) extends to all $u\in L^{1}_{\loc}(\R^N)$ thanks to (\ref{liminf_est_nguyen}).

\item (Trivial Gamma-limit when $\gamma\in (-1,0)$ \cite{Nguyen-survey}) For every $p\geq 1$ and $\gamma \in (-1,0)$ it holds that
$$\Gamma-\lim_{\lambda\to 0^+} \Fgpl(u,\R^N)=0 \qquad \forall u\in L^1 _{\loc}(\R^N).$$

This is due to the fact that, in this case, it holds that
\begin{equation}\label{null-limit-pc}
\lim_{\lambda\to 0^+} \Fgpl(w,\R^N)=0
\end{equation}
whenever $w:\R^N \to \R$ is a piecewise constant function with compact support.
\end{itemize}

We conclude our overview of the literature by mentioning that extensions of some of the above results in various directions, including applications to image processing, approximations of more general functionals, or in more general settings, can be found in \cite{2018-CRAS-AGMP,2020-APDE-AGP,BP-metric,2018-AnnPDE-BN,2020-CCM-BN,CLL-Heisenberg,max_funct_metric,DS-metric,DM-operators,HLYY-high_order,weak-type_metric,NS-anisotropic,NPSV-magnetic,2023-CalcVar-ZYY,2023-CCM-ZYY}.

\subsection{New results}

The first aim of this paper is to fill some of the gaps in the case $\gamma<-p$.

Indeed, our first main result is the extension of the Gamma-liminf estimate obtained in \cite{GP-gammaliminf} to this regime, that correspond to the following statement.

\begin{thm}[Gamma-liminf estimate for $\gamma<-p$]\label{thm:gamma-liminf}
Let $p\geq 1$ and $\gamma<-p$ be real numbers, let $N$ be a positive integer, and let $\Omega\subseteq \R^N$ be any open set. 

Then it holds that
\begin{equation}
\Gamma\mbox{--}\liminf_{\lambda\to 0^+} \Fgpl(u,\Omega) \geq 
C_{N,p}\cdot 2^{\gamma}\log 2 \cdot F_p(u,\Omega) 
\qquad 
\forall u \in L^1 _{\loc} (\Omega),
\end{equation}
where $C_{N,p}$ is the geometric constant defined by (\ref{defn:C_N}).

More precisely, for every $u\in L^1 _{\loc}(\Omega)$ and every family of functions $\{u_\lambda\}\subseteq L^1_{\loc}(\Omega)$ such that $u_\lambda \to u$ in $L^1 _{\loc}(\Omega)$ as $\lambda \to 0^+$, it holds that
\begin{equation}\label{th:liminf_ineq}
\liminf_{\lambda\to 0^+} \Fgpl(u_\lambda,\Omega) \geq 
C_{N,p}\cdot 2^\gamma \log2 \cdot F_p(u,\Omega).
\end{equation}
\end{thm}

We point out that, as in \cite{GP-gammaliminf}, the statement of Theorem~\ref{thm:gamma-liminf} remains true if one replaces $L^1 _{\loc}$ convergence with other notions (see Remark~\ref{rmk:convergence}).

Moreover, by taking the trivial family $u_\lambda \equiv u$, we immediately get the following corollary, which extends (\ref{lim-sobolev}) to all $u\in L^1 _{\loc}(\R^N)$ also in the case $\gamma<-p$.

\begin{cor}\label{cor:liminf}
For every $p\geq 1$ and $\gamma<-p$, and for every open set $\Omega \subset \R^N$ it holds that
$$\liminf_{\lambda\to 0^+} \Fgpl(u,\Omega) \geq 
C_{N,p}\cdot 2^\gamma \log 2 \cdot F_p(u,\Omega)\qquad \forall u \in L^1 _{\loc} (\Omega).$$
\end{cor}

This settles the only remaining case in the open problems stated in \cite[Section~7.2]{BSVY} and \cite[Open question~14]{Nguyen-survey}. Indeed, combining Corollary~\ref{cor:liminf} with the results collected in the previous section, we now know that for every $p\geq 1$ and every $\gamma\in\R\setminus (-1,0]$ there is some positive constant $c_0(N,\gamma,p)>0$ such that
\begin{gather*}
\liminf_{\lambda\to +\infty} \Fgpl(u,\R^N) \geq c_0(N,\gamma,p)\cdot F_p(u,\R^N)\qquad \text{if }\gamma>0,\\
\liminf_{\lambda\to 0^+} \Fgpl(u,\R^N) \geq c_0(N,\gamma,p)\cdot F_p(u,\R^N)\qquad \text{if }\gamma\leq -1,
\end{gather*}
for every $u\in L^1_{\loc}(\R^N)$. On the other hand, for every $\gamma\in (-1,0)$ such statement fails because of (\ref{null-limit-pc}).

\smallskip

Our second goal is to investigate the validity of a Poincaré-type inequality like (\ref{nguyen-poincaré}) in the more general case $\gamma\in \R \setminus (-1,0]$, which is a question posed by H.~M.~Nguyen in \cite[Open question~33]{Nguyen-survey}. In the case $\gamma<-p$, we establish such inequality in every bounded convex domain.

\begin{thm}[Poincaré inequality for $\gamma<-p$] \label{thm:poincaré_<-p}
Let $p\geq 1$ and $\gamma<-p$ be real numbers, let $N$ be a positive integer, and let $\Omega\subseteq \R^N$ be a bounded and convex open set.

Then there exists a positive constant $K(\Omega,N,p,\gamma)>0$ such that
\begin{equation}\label{th:poincaré_<-p}
\iint_{\Omega \times \Omega} |u(y)-u(x)|^p\,dxdy \leq K(\Omega,N,\gamma,p)\cdot \left(\Fgpl(u,\Omega) + \lambda ^p \right),    
\end{equation}
for every $u\in L^p(\Omega)$ and every $\lambda>0$.
\end{thm}

As for the case $\gamma>0$, we provide a counterexample showing that such inequality cannot hold without some correction or additional assumption (see Proposition~\ref{prop:counterex_>0}). However, we prove the following modified version, which includes a term depending on the oscillation of $u$, that is the difference between its (essential) supremum and infimum.

\begin{thm}[Poincaré inequality for $\gamma>0$] \label{thm:poincaré_>0}
Let $p\geq 1$ and $\gamma>0$ be real numbers, let $N$ be a positive integer, and let $\Omega\subseteq \R^N$ be a bounded and convex open set.

Then there exists a positive constant $K(\Omega,N,p,\gamma)>0$ such that
$$\iint_{\Omega \times \Omega} |u(y)-u(x)|^p\,dxdy \leq K(\Omega,N,\gamma,p)\cdot \left(\Fgpl(u,\Omega) + \lambda ^{-p} \osc(u,\Omega)^{2p} \right),$$
for every $u\in L^\infty(\Omega)$ and every $\lambda>0$.
\end{thm}

We remark that, even though estimating the $L^p$ oscillation with an $L^\infty$ oscillation term might seem trivial, this inequality is still meaningful when $\lambda$ is sufficiently large, and it recovers the local Poincaré inequality in the limit as $\lambda \to +\infty$.

\paragraph{Overview of the technique}

The proof of Theorem~\ref{thm:gamma-liminf} relies on the strategy developed in \cite{GP-gammaliminf} (which in turn was inspired by \cite{1998-CPAM-Gobbino}). We begin by considering a sort of cell problem in the one-dimensional case, namely we provide a quantitative estimate of the minimal energy of a function $u:(a,b)\to [A,B]$ such that $u(x)=A$ for $x$ close to $a$, and $u(x)=B$ for $x$ close to $b$ (see Proposition~\ref{prop:gatto}).

As in \cite{GP-gammaliminf}, this estimate is achieved by first representing the functional (\ref{defn:Fgpl}) as an average, over $\delta\in (1/2,1)$ and $x\in (0,\delta)$, of a discretized version of the functional which takes into account only the values of $u$ on points of the form $x+\delta i 2^{-k}$, with $i\in \Z$ and $k\in\N$. In this way we reduce the cell problem to a combinatorial estimate in the dyadic setting, that we prove in Proposition~\ref{prop:est_DF}.

After this local estimate is established, we prove the one-dimensional version of the liminf inequality by standard localization arguments, and we extend it to higher dimensions with the usual slicing method.

As for the Poincaré inequality, we still start with the one-dimensional case $\Omega=(a,b)$, and we rely again on the combination of the dyadic discretization and the estimate in the discrete setting, in order to bound the $L^p$ oscillation with the functional. Here the main difficulty is due to the fact that, when $\gamma<-p$, the discrete estimate in an interval $[x,y]\subset (a,b)$ holds only if the couple $(x,y)$ made by the endpoints belongs to the set $\Egpl(u,(a,b))$, namely if $|u(y)-u(x)|>\lambda |y-x|^{1+\gamma/p}$.

Roughly speaking, we solve this problem by showing that we can compute the double integral
$$\iint_{\Omega\times\Omega}|u(y)-u(x)|^p\,dxdy$$
just by integrating over couples $(x,y)$ for which both $|u(y)-u(x)|$ and $|y-x|$ are not too small (see Lemma~\ref{lemma:meas_theory}). Then, if $\lambda$ is small enough, all such couples $(x,y)$ belong to $\Egpl(u,(a,b))$, so we can apply the discrete estimate. On the other hand, if $\lambda$ is not small, then we can exploit term $K\cdot\lambda^p$ in the right-hand side. We remark that the inequality cannot be true without such addendum, since it holds that
$$\Fgpl(u,(a,b))=0 \qquad \forall\lambda>\frac{\osc(u,(a,b))}{(b-a)^{1+\gamma/p}}.$$

In the case $\gamma>0$ this approach does not work. In fact, in this case the condition to apply the discrete estimate becomes that $(x,y)$ does \emph{not} belong to $\Egpl(u,(a,b))$. However, when $\gamma>0$ we cannot deduce that $(x,y)\notin \Egpl(u,(a,b))$ if we just know that $|y-x|$ is not close to zero, unless we know in advance that the oscillation of $u$ is small with respect to $\lambda$. This is the reason for which we need to add the oscillation in the right-hand side.

\paragraph{Structure of the paper}

The paper is organized as follows. In section~\ref{sec:gatto} we prove representation formula and the lower bound for the dyadic functional in the case $\gamma<-p$. In section~\ref{sec:gamma-liminf} we prove Theorem~\ref{thm:gamma-liminf}. In section~\ref{sec:Poincaré} we prove Theorem~\ref{thm:poincaré_<-p} and Theorem~\ref{thm:poincaré_>0}, and we show that in the case $\gamma>0$ the Poincaré inequality cannot hold without some correction. Finally, in section~\ref{sec:op} we discuss some open problems.

We conclude the paper with a short appendix including a result that we do not use in this paper, but that we hope might be useful for future research.

\section{The dyadic functional}\label{sec:gatto}

In this section we extend the dyadic discretization introduced in \cite{GP-gammaliminf} to every $\gamma\in\R$ and we prove a lower bound for the discretized version of the functional. To this end, let us recall the notation used in \cite{GP-gammaliminf}, with some small modifications.

Let $\gamma \in \R$, $p\geq 1$, and $\lambda>0$ be real numbers. Let us consider the function
\begin{equation}
\psigpl(\delta,\Delta):=
\begin{cases}
1 &\mbox{if }\Delta > \lambda \delta^{1+\gamma/p}\\
0 &\mbox{if }\Delta \leq \lambda \delta^{1+\gamma/p}
\end{cases}
\qquad\quad
\forall(\delta, \Delta)\in [0,+\infty)^2.
\label{defn:psi-gpl}
\end{equation}

Let $\D:=\{i/2^{k}:i\in \Z, \ k\in\N \}$ denote the set of \emph{dyadic numbers}. For every nonnegative integer $k\in\N$, a \emph{$k$-dyadic interval} is any interval of the form
\begin{equation}\label{defn:D-int}
\left[\frac{i}{2^k},\frac{i+1}{2^k}\right]
\end{equation}
for some $i\in \Z$. More generally, a \emph{dyadic interval} is a closed interval that is a $k$-dyadic interval for some $k\in\N$.

For every interval $I\subseteq\R$ and every integer $k\geq 0$, we consider the $k$-dyadic intervals contained in $I$, which are indexed by the set
\begin{equation}\label{defn:Dk}
\D_k(I):=\left\{i\in \Z: \left[\frac{i}{2^k},\frac{i+1}{2^k}\right]\subseteq I \right\}.
\end{equation}

For every $\delta>0$ and every function $v:\D\cap I \to \R$, the \emph{dyadic functional} is defined as
\begin{equation}\label{defn:DF}
\DFgpl(\delta,v,I):=\sum_{k=0} ^{\infty} \frac{1}{2^{k(1+\gamma)}} \sum_{i\in \D_k(I)} \psigpl\left(\frac{\delta}{2^k}, \left|v\left(\frac{i+1}{2^k}\right)-v\left(\frac{i}{2^k}\right)\right| \right).
\end{equation}

We are now ready to state the two crucial properties of the dyadic functional. The first one is the following representation formula, which extends \cite[Proposition~2.1]{GP-gammaliminf} to every $\gamma\in\R$, at the cost of replacing an equality with an inequality.

\begin{prop}[Dyadic representation]\label{prop:dyadic_repr}

Let $\gamma\in\R$ and $p\geq 1$ be real numbers. Let $(a,b)\subseteq\R$ be an interval, and let $u:(a,b)\to \R$ be a measurable function. 

For every $x\in \R$, and every real number $\delta>0$, let us consider the numbers
\begin{equation}
a_{x,\delta}:=\frac{a-x}{\delta}
\qquad\text{and}\qquad
b_{x,\delta}:=\frac{b-x}{\delta},
\label{defn:abxd}
\end{equation}
and the function $v_{x,\delta}:\D\cap (a_{x,\delta},b_{x,\delta})\to \R$ defined by
\begin{equation}\label{defn:v_x,delta}
v_{x,\delta}(q) :=u(x+q\delta) \qquad \forall q\in \D\cap (a_{x,\delta}, b_{x,\delta}).
\end{equation}

Then for every positive real number $\lambda>0$ it turns out that
\begin{equation}\label{th:dyadic_repr}
\Fgpl(u,(a,b))\geq 2\lambda^p \int_{1/2} ^1  d\delta \, \delta^{\gamma-1} \int_{0} ^{\delta} \DFgpl\left(\delta,v_{x,\delta},(a_{x,\delta}, b_{x,\delta}) \right)\,dx.
\end{equation}
\end{prop}

\begin{proof}
The proof is identical to the proof of \cite[Proposition~2.1]{GP-gammaliminf}. The additional assumptions in \cite{GP-gammaliminf}, including the positivity of $\gamma$, were only used to ensure that no interactions, namely points in $\Egpl(u,(a,b))$, were neglected in the discretization procedure, so that (\ref{th:dyadic_repr}) was an equality in that case.

However, even if we neglected some interactions, that would only mean that the discretization might reduce the energy, so we still deduce the correct inequality.
\end{proof}

The next result is an estimate from below for the dyadic functional in terms of the oscillation of $v$ in $(a,b)$, which extends \cite[Proposition~2.2]{GP-gammaliminf} to the case $\gamma<-p$. The strategy of the proof is similar to the one in \cite{GP-gammaliminf} for the case $\gamma>0$, with suitable adaptations.

The general idea is that here the dyadic intervals contributing to the sum in (\ref{defn:DF}), which we call \emph{visible intervals}, are the larger ones, as opposite to the case $\gamma>0$, in which small intervals were relevant (this is due to the sign of the exponent in (\ref{defn:Egpl})). However, in both cases the crucial point is to identify those \emph{important intervals} that separate the visible intervals form the rest.

\begin{prop}[Asymptotic cost of oscillations for the dyadic functional]\label{prop:est_DF}
Let $p\geq 1$, $\gamma<-p$, and $B\geq A$ be real numbers. For a natural number $k_0\in\N$, let $[\alpha,\beta]\subseteq\R$ be a $k_0$-dyadic interval, and let $v:\D\cap [\alpha,\beta] \to \R$ be a function such that $v(\alpha)=A$ and $v(\beta)=B$.

Then for every $\delta\in[1/2,1]$, and every positive real number $\lambda< 2^{(k_0+1)(1+\gamma/p)}(B-A)$, it holds that
\begin{equation}\label{th:est_DF}
\DFgpl(\delta,v,[\alpha,\beta])\geq 
2^\gamma \cdot\frac{1}{\lambda^p \delta^{p+\gamma}}
\cdot \frac{(B-A)^p}{(\beta-\alpha)^{p-1}}.
\end{equation}
\end{prop}

\begin{proof}[Proof of Proposition~\ref{prop:est_DF}]
We observe that the functional (\ref{defn:DF}) does not increase under truncation, namely if we replace $v(x)$ with $\max\{\min\{v(x),B\},A\}$, so we can assume without loss of generality that
\begin{equation}\label{hp:v_bounded}
A\leq v(x) \leq B \qquad \forall x\in \D\cap [\alpha,\beta].
\end{equation}

Let us recall some notation and terminology from \cite{GP-gammaliminf}. We say that two intervals are \emph{essentially disjoint} if they do not share internal points, namely if they intersect at most at the boundary. We point out that any two dyadic intervals are either essentially disjoint or they are one contained into the other.

We say that a $k$-dyadic interval of the form (\ref{defn:D-int}) is \emph{visible} if
\begin{equation*}
\left|v\left(\frac{i+1}{2^k}\right) - v\left(\frac{i}{2^k}\right) \right| >\lambda \left(\frac{\delta}{2^{k}}\right)^{1+\gamma/p},
\end{equation*}
and we consider the set
\begin{equation*}
    \Vis_k:=
    \left\{
    i\in\Z:
    \left[\frac{i}{2^k},\frac{i+1}{2^k}\right]\subseteq[\alpha,\beta] 
    \mbox{ is visible} 
    \right\},
\end{equation*}
that indexes all $k$-dyadic intervals that are visible and contained in $[\alpha,\beta]$.

Visible dyadic intervals are exactly those whose endpoints give a positive contribution in the right-hand side of (\ref{defn:DF}), so that
\begin{equation}
\DFgpl(\delta,v,[\alpha,\beta])=
\sum_{k=k_0}^\infty \frac{|\Vis_k|}{2^{k(1+\gamma)}},
\label{est:DF-Vis}
\end{equation}
where the bars denote the number of elements of a set.

Now, for every $k\geq k_0$, we introduce the set of $k$-dyadic intervals that are \emph{important}, but we point out that this definition is different from the one used in \cite{GP-gammaliminf}.

The definition is recursive in $k$: we say that a $k$-dyadic interval $I$ is important if it verifies the following two conditions.
\begin{enumerate}
\renewcommand{\labelenumi}{(\roman{enumi})}
\item It is not visible.

\item The unique $(k-1)$-dyadic interval that contains $I$ is visible.

\item It is not contained in any important $h$-dyadic interval with $h\leq k-1$.
\end{enumerate}

We remark that the assumptions imply that
\begin{equation}\label{[alpha,beta]-visible}
|v(\beta)-v(\alpha)|=B-A>\lambda \left(\frac{1}{2^{k_0+1}}\right)^{1+\gamma/p}\geq \lambda \left(\frac{\delta}{2^{k_0}}\right)^{1+\gamma/p},
\end{equation}
and this means that $[\alpha,\beta]$, which is the only $k_0$-dyadic interval contained in $[\alpha,\beta]$, is visible, hence it is not important, and in particular there are no important $k_0$-dyadic intervals.

We also introduce the set
\begin{equation*}
    \Imp_k:=
    \left\{
    i\in\Z:
    \left[\frac{i}{2^k},\frac{i+1}{2^k}\right]\subseteq[\alpha,\beta] 
    \mbox{ is important} 
    \right\},
\end{equation*}
that indexes all $k$-dyadic intervals that are important and contained in $[\alpha,\beta]$.

In the sequel we need the following properties.
\begin{enumerate}[label={(\arabic*)}]

\item Important dyadic intervals (even with different $k$'s) are pairwise essentially disjoint.

\item There are finitely many important dyadic intervals, and they cover $[\alpha,\beta]$.
\end{enumerate}

Let us prove these two properties.

The first property is immediate: since property (iii) prevents important dyadic intervals to be contained one into another, then they must be pairwise essentially disjoint.

As for the second one, we observe that, since $\delta\leq 1$ and $1+\gamma/p<0$, then from (\ref{hp:v_bounded}) we deduce that that
\begin{equation}\label{small-interval-invisible}
\frac{1}{2^k}<\left(\frac{B-A}{\lambda}\right)^{\frac{1}{1+\gamma/p}} \Rightarrow \left|v\left(\frac{i+1}{2^k}\right)-v\left(\frac{i}{2^k}\right)\right|\leq B-A < \frac{\lambda}{2^{k(1+\gamma/p)}}\leq \lambda \left(\frac{\delta}{2^{k}}\right)^{1+\gamma/p},
\end{equation}
and this means that there exists some integer $K>k_0$ such that $k$-dyadic intervals cannot be visible if $k\geq K$. As a consequence, by property (ii), $k$-dyadic intervals cannot be important if $k\geq K+1$, and hence there cannot be infinitely many important intervals.

Moreover, we recall that $[\alpha,\beta]$ is visible because of (\ref{[alpha,beta]-visible}). Therefore, if we fix any $x\in [\alpha,\beta]$, then we can consider the largest dyadic interval $I$ contained in $[\alpha,\beta]$ that contains $x$ and is not visible. This might be non unique if $x$ is a dyadic number, but in this case we just consider any of the two such intervals. Then $I$ is a $k$-dyadic interval for some $k>k_0$, it is not visible, and all $h$-dyadic intervals with $h<k$ that are contained in $[\alpha,\beta]$ and contain $I$ are visible, and hence not important. This is enough to deduce that $I$ is important, and hence we deduce that important dyadic intervals cover $[\alpha,\beta]$.

\smallskip

Now we claim that
\begin{equation}\label{est_J}
(B-A)^p\leq \lambda^p\delta^{p+\gamma}\cdot (\beta-\alpha)^{p-1}\cdot \sum_{k=k_0+1}^{K}\frac{|\Imp_k|}{2^{k(1+\gamma)}}.
\end{equation}

In order to prove this, we observe that, thanks to property (2) of important dyadic interval, we can write
\begin{equation}\label{union-imp}
[\alpha,\beta]= 
\bigcup_{k=k_0+1} ^{K} \bigcup_{i\in \Imp_k} \left[\frac{i}{2^k},\frac{i+1}{2^k}\right],
\end{equation}
and the intervals appearing in the right-hand side are pairwise essentially disjoint.

Since all these intervals are not visible by definition, we can estimate the difference between the values on $v$ at the endpoints of each interval in the following way
\begin{multline*}
B-A = |v(\beta)-v(\alpha)|
\leq  \sum_{k=k_0+1} ^{K}\sum_{i\in \Imp_k} \left|v\left(\frac{i+1}{2^k}\right) - v\left(\frac{i}{2^k}\right)\right|\\
\leq  \sum_{k=k_0+1} ^{K}\sum_{i\in \Imp_k} \lambda\left(\frac{\delta}{2^k}\right)^{1+\gamma/p}
=  \sum_{k=k_0+1} ^{K} \lambda |\Imp_k| \left(\frac{\delta}{2^k}\right)^{1+\gamma/p}.
\end{multline*}

At this point by Hölder inequality and (\ref{union-imp}) we conclude that
\begin{align*}
B-A &\leq  \lambda \delta^{1+\gamma/p} \left(\sum_{k=k_0+1}^{K} \frac{|\Imp_k|}{2^{k(1+\gamma)}}\right)^{1/p} \left(\sum_{k=k_0+1}^{K} \frac{|\Imp_k|}{2^k}\right)^{(p-1)/p}\\
&= \lambda \delta^{1+\gamma/p} \left(\sum_{k=k_0+1}^{K} \frac{|\Imp_k|}{2^{k(1+\gamma)}}\right)^{1/p} (\beta-\alpha)^{(p-1)/p},
\end{align*}
which is equivalent to (\ref{est_J}).

\smallskip

Now we observe that
\begin{equation}\label{est:vis_imp}
|\Vis_{k-1}| \geq \frac{1}{2} |\Imp_{k}|,
\end{equation}
because every important $k$-dyadic interval is contained in a visible $(k-1)$-dyadic interval, and any $(k-1)$-dyadic interval can contain at most two (important) $k$-dyadic intervals.

Thus, from (\ref{est:DF-Vis}) we deduce that
\begin{equation*}
\DFgpl(\delta,v,[\alpha,\beta]])\geq\sum_{k=k_0}^{K-1} \frac{|\Vis_k|}{2^{k(1+\gamma)}} =2^{1+\gamma}\sum_{k=k_0+1}^{K} \frac{|\Vis_{k-1}|}{2^{k(1+\gamma)}} \geq 2^\gamma\sum_{k=k_0+1}^{K} \frac{|\Imp_k|}{2^{k(1+\gamma)}} .
\end{equation*}

Recalling (\ref{est_J}), this inequality implies (\ref{th:est_DF}).
\end{proof}

We may now remove the assumption that $[\alpha,\beta]$ is precisely a dyadic interval, at the cost of introducing an error term.

\begin{cor}\label{cor:est_DF}
Let $\alpha<\beta$ be dyadic numbers such that $\alpha=i_1/2^{k_0}$ and $\beta=i_2/2^{k_0}$ for some integers $i_1<i_2$ and for some number $k_0\in\N$. Let also $p\geq 1$, $\gamma<-p$, and $B\geq A$ be real numbers, and let $v:\D\cap [\alpha,\beta] \to [A,B]$ be a function such that $v(\alpha)=A$ and $v(\beta)=B$.

Then for every $\delta\in[1/2,1]$, and every real number $\lambda>0$ such that
$$(i_2-i_1)\lambda 2^{-(k_0+1)(1+\gamma/p)}< B-A,$$
it holds that
\begin{equation}\label{th:est_DF_error}
\DFgpl(\delta,v,[\alpha,\beta])\geq 
2^\gamma \cdot\frac{1}{\lambda^p \delta^{p+\gamma}}
\cdot \frac{(B-A-(i_2-i_1)\lambda 2^{-(k_0+1)(1+\gamma/p)} )^p}{(\beta-\alpha)^{p-1}}.
\end{equation}
\end{cor}

\begin{proof}
First of all, we observe that 
$$\DFgpl(\delta,v,[\alpha,\beta]])=\sum_{i=i_1} ^{i_2-1} \DFgpl(\delta,v,[i/2^{k_0},(i+1)/2^{k_0}]]).$$

Now we can apply (\ref{th:est_DF}) in all the $k_0$-dyadic intervals in the right-hand side for which $\lambda< 2^{(k_0+1)(1+\gamma/p)}|v((i+1)/2^{k_0})-v(i/2^{k_0})|$, so we obtain that
$$\DFgpl(\delta,v,[\alpha,\beta]])\geq \sum_{i=i_1} ^{i_2-1} 2^{\gamma}\cdot\frac{1}{\lambda^p\delta^{p+\gamma}}\cdot\frac{|v((i+1)/2^{k_0})-v(i/2^{k_0})|^p}{2^{-k_0(p-1)}} \chi(i),$$
where
$$\chi(i)=\begin{cases}
1&\text{if }\lambda< 2^{(k_0+1)(1+\gamma/p)}|v((i+1)/2^{k_0})-v(i/2^{k_0})|,\\
0&\text{if }\lambda\geq 2^{(k_0+1)(1+\gamma/p)}|v((i+1)/2^{k_0})-v(i/2^{k_0})|.
\end{cases}$$

Let us set $N:=|\{i\in\{i_1,\dots i_2\}:\chi(i)=1\}|$, so Jensen inequality yields
\begin{align*}
\sum_{i=i_1} ^{i_2-1} \frac{|v((i+1)/2^{k_0}-v(i/2^{k_0})|^p}{2^{-k_0(p-1)}}\chi(i)
&=\sum_{i=i_1} ^{i_2-1} \left(\frac{|v((i+1)/2^{k_0})-v(i/2^{k_0})|}{2^{-k_0}}\right)^p 2^{-k_0}\chi(i)\\
&\geq \left(\sum_{i=i_1} ^{i_2-1} \frac{|v((i+1)/2^{k_0})-v(i/2^{k_0})|}{N 2^{-k_0}}\chi(i)\right)^p N 2^{-k_0}\\
&=\frac{\left(\sum_{i=i_1} ^{i_2-1} |v((i+1)/2^{k_0})-v(i/2^{k_0})|\chi(i)\right)^p}{(N 2^{-k_0})^{p-1}},
\end{align*}
and we conclude by observing that $N2^{-k_0}\leq (i_2-i_1)2^{-k_0}=\beta-\alpha$ and
\begin{align*}
\sum_{i=i_1} ^{i_2-1} |v((i+1)/2^{k_0})-v(i/2^{k_0})|\chi(i)
&\geq \sum_{i=i_1} ^{i_2-1} \left\{|v((i+1)/2^{k_0})-v(i/2^{k_0})| - \lambda 2^{-(k_0+1)(1+\gamma/p)}\right\}\\
&\geq B-A-(i_2-i_1)\lambda 2^{-(k_0+1)(1+\gamma/p)}.
\end{align*}

Combining all the previous estimates we get (\ref{th:est_DF_error}).
\end{proof}

\section{Gamma-liminf estimate: proof of Theorem~\ref{thm:gamma-liminf}}\label{sec:gamma-liminf}

In this section we exploit the dyadic theory of section~\ref{sec:gatto} to prove Theorem~\ref{thm:gamma-liminf}. The strategy is quite classical, and very close to \cite[Section~3]{GP-gammaliminf}, so we limit ourselves to provide sketches of the proofs, in order to avoid excessive copy-paste from \cite{GP-gammaliminf}.

In the rest of the paper, we denote by $\Leb^N$ the $N$-dimensional Lebesgue measure.

\begin{prop}[Lower bound for the cell problem]\label{prop:gatto}

Let $p\geq 1$, $\gamma<-p$, and $B\geq A$ be real numbers. Let $k_0\in\N$ be a natural number, let $(a,b)\subseteq\R$ be an interval with length $b-a\geq 3 \cdot 2^{-k_0}$, and let $u:(a,b)\to [A,B]$ be a measurable function. 

Let us assume that there exists a real number $\ep>0$ such that
\begin{equation}\label{hp:density_ep}
\Leb^1\bigl(\{ x\in (a,a+2^{-k_0}): u(x)\neq A \}\bigr)+
\Leb^1\bigl(\{ x\in (b-2^{-k_0},b): u(x)\neq B \}\bigr)\leq 
\frac{\ep}{2}\cdot 2^{-k_0}.
\end{equation}

Then for every real number $\lambda>0$ such that
$$\lambda 2^{k_0}(b-a)2^{-(k_0+1)(1+\gamma/p)} < B-A,$$
it holds that
\begin{equation}\label{th:gatto}
\Fgpl (u,(a,b))\geq (1-\ep) 2^{1+\gamma} \log 2 \cdot\frac{(B-A-\lambda 2^{k_0}(b-a)2^{-(k_0+1)(1+\gamma/p)})^p}{(b-a)^{p-1}},
\end{equation}
\end{prop}

\begin{proof}
Let us fix $\delta \in [1/2,1]$ and $x\in (0,\delta)$. Let us consider the interval $(a_{x,\delta},b_{x,\delta})$ whose endpoints are defined by (\ref{defn:abxd}), and let us denote by $[\alpha_{x,\delta},\beta_{x,\delta}]$ the largest closed interval contained in $(a_{x,\delta},b_{x,\delta})$ that can be written as $[i_1/2^{k_0},i_2/2^{k_0}]$ for some integers $i_1<i_2$, namely the interval with endpoints
\begin{equation*}
    \alpha_{x,\delta}:=2^{-k_0}\lfloor 2^{k_0} a_{x,\delta}\rfloor+2^{-k_0}
    \qquad\text{and}\qquad
    \beta_{x,\delta}:=2^{-k_0}\lceil 2^{k_0}b_{x,\delta}\rceil-2^{-k_0}.
\end{equation*}
Note that $\beta_{x,\delta}\geq\alpha_{x,\delta}+2^{-k_0}$ because we assumed that $b-a\geq 3\cdot 2^{-k_0}$, and that
\begin{equation}\label{est:i_2-i_1}
    i_2-i_1\leq 2^{k_0}(b-a)
\end{equation}

Let us set also $v_{x,\delta}(y):=u(x+ y\delta)$ as in (\ref{defn:v_x,delta}) and let $X_\delta$ denote the sets of points $x\in(0,\delta)$ such that
\begin{equation*}
v_{x,\delta}(\alpha_{x,\delta})=u(x+\alpha_{x,\delta}\delta)=A
\qquad \text{and} \qquad
v_{x,\delta}(\beta_{x,\delta})=u(x+\beta_{x,\delta}\delta)=B.
\end{equation*}

We observe that the two maps $\psi_1(x):=x+\alpha_{x,\delta}\delta$ and $\psi_2(x):=x+\beta_{x,\delta}\delta$ are measure preserving from $(0,2^{-k_0}\delta)$ to $(a,a+2^{-k_0}\delta)$ and from $(0,2^{-k_0}\delta)$ to $(b-2^{-k_0}\delta,b)$, respectively, and are both $2^{-k_0}\delta$-periodic. Thus, we deduce that
\begin{equation*}
    \Leb^1\bigl(\{x\in(0,\delta):v_{x,\delta}(\alpha_{x,\delta})\neq A\}\bigr)= 2^{k_0}
    \Leb^1\bigl(\{y\in(a,a+2^{-k_0}\delta):u(y)\neq A\}\bigr),
\end{equation*}
and
\begin{equation*}
    \Leb^1\bigl(\{x\in(0, \delta):v_{x,\delta}(\beta_{x,\delta})\neq B\}\bigr)=2^{k_0}
    \Leb^1\bigl(\{y\in(b-2^{-k_0}\delta,b):u(y)\neq B\}\bigr).
\end{equation*}

Therefore, since
\begin{equation*}
    (0,\delta)\setminus X_{\delta}=
    \left\{x\in(0,\delta):v_{x,\delta}(\alpha_{x,\delta})\neq A\right\}\cup
    \left\{x\in(0,\delta):v_{x,\delta}(\beta_{x,\delta})\neq B\right\},
\end{equation*}
from assumption (\ref{hp:density_ep}) we deduce that
\begin{align*}
    \lefteqn{\Leb^1\bigl((0,\delta)\setminus X_{\delta}\bigr)}\\
    &\quad \leq 2^{k_0} 
    \Leb^1\bigl(\{y\in(a,a+2^{-k_0}\delta):u(y)\neq A\}\cup \{y\in(b-2^{-k_0}\delta,b):u(y)\neq B\}\bigr)
    \\
    &\quad \leq 
    2^{k_0} \Leb^1\bigl(\{y\in(a,a+2^{-k_0}):u(y)\neq A\}\bigr)+
    2^{k_0} \Leb^1\bigl(\{y\in(b-2^{-k_0},b):u(y)\neq B\}\bigr)
    \\
    &\quad \leq 
    \frac{\ep}{2},
\end{align*}
and hence, recalling that $\delta\geq 1/2$,
\begin{equation}\label{leb_Xdelta}
\Leb^1(X_\delta)\geq \delta - \frac{\ep}{2} \geq (1-\ep)\delta.
\end{equation}

Now, keeping (\ref{est:i_2-i_1}) into account, we observe that for every $x$ in $X_\delta$ the function $v_{x,\delta}$ and the parameter $\lambda$ satisfy the assumptions of Corollary~\ref{cor:est_DF} with $[\alpha,\beta]:=[\alpha_{x,\delta},\beta_{x,\delta}]$, so we obtain that
\begin{equation}
\DFgpl(\delta,v_{x,\delta},[\alpha_{x,\delta},\beta_{x,\delta}])\geq 
2^\gamma \cdot\frac{1}{\lambda^p \delta^{p+\gamma}}
\cdot\frac{(B-A-\lambda 2^{k_0}(b-a)2^{-(k_0+1)(1+\gamma/p)})^p}{(\beta_{x,\delta}-\alpha_{x,\delta})^{p-1}}.
\label{est:DF-abxd}
\end{equation}

Since
\begin{equation*}
\beta_{x,\delta}-\alpha_{x,\delta}\leq 
b_{x,\delta}-a_{x,\delta} = 
\frac{b-a}{\delta},
\end{equation*}
from (\ref{est:DF-abxd}) we conclude that 
\begin{equation*}
\DFgpl(\delta,v_{x,\delta},(a_{x,\delta},b_{x,\delta}))\geq 
2^\gamma \cdot\frac{1}{\lambda^p \delta^{1+\gamma}}
\cdot\frac{(B-A-\lambda 2^{k_0}(b-a)2^{-(k_0+1)(1+\gamma/p)})^p}{(b-a)^{p-1}}
\end{equation*}
for every $\delta\in[1/2,1]$ and every $x\in X_\delta$. At this point, we exploit Proposition~\ref{prop:dyadic_repr} so we obtain that
\begin{align*}
\Fgpl(u,(a,b)) & \geq  
2\lambda^p \int_{1/2} ^1  d\delta \, \delta^{\gamma-1} \int_{0} ^{\delta} \DFgpl\left(\delta,v_{x,\delta},(a_{x,\delta}, b_{x,\delta}) \right)\,dx
\\
& \geq 
2\lambda^p \int_{1/2} ^1  d\delta \, \delta^{\gamma-1} \int_{X_\delta} \DFgpl\left(\delta,v_{x,\delta},(a_{x,\delta}, b_{x,\delta}) \right)\,dx
\\
& \geq 
2^{1+\gamma}\cdot\frac{(B-A-\lambda 2^{k_0}(b-a)2^{-(k_0+1)(1+\gamma/p)})^p}{(b-a)^{p-1}} 
\int_{1/2}^{1}\frac{\Leb^1(X_\delta)}{\delta^2}\,d\delta.
\end{align*}

Recalling (\ref{leb_Xdelta}) we get that
$$\int_{1/2}^{1}\frac{\Leb^1(X_\delta)}{\delta^2}\,d\delta\geq (1-\ep)\int_{1/2}^1 \frac{d\delta}{\delta}=(1-\ep)\log 2,$$
so we obtain exactly (\ref{th:gatto}).
\end{proof}

At this point, the proof of Theorem~\ref{thm:gamma-liminf} is rather standard.

\begin{proof}[Proof of Theorem~\ref{thm:gamma-liminf}]

Let us start with the proof in the case $N=1$ and $\Omega=(a,b)$. To this end, we consider any function $u \in L^1 _{\loc}((a,b))$, and any family $\{u_{\lambda}\} \subseteq L^1 _{\loc} ((a,b))$ such that $u_{\lambda}\to u$ in $L^1 _{\loc} ((a,b))$ as $\lambda\to +\infty$. Since $C_{1,p}=2$ for every $p\geq 1$, in this case we need to prove that
\begin{equation}
\liminf_{\lambda\to +\infty} \Fgpl(u_\lambda,(a,b))\geq 2 \cdot 2^\gamma\log 2\cdot F_p(u,(a,b)),
\label{th:1D}
\end{equation}
that is (\ref{th:liminf_ineq}) in this case.

Let $c<d$ be Lebesgue points of $u$ in $(a,b)$. Then we claim that
\begin{equation}\label{liminf_leb_pts}
\liminf_{\lambda\to +\infty} \Fgpl(u_{\lambda},(c,d)) \geq
2^{1+\gamma}\log 2 \cdot \frac{|u(d)-u(c)|^p}{(d-c)^{p-1}}.
\end{equation}

To prove this, we can assume without loss of generality that $u(c)<u(d)$ so that, if we fix some small $\ep>0$ and we replace every $u_\lambda$ with its truncation
$$u_{\lambda,\ep}(x):=\max\{\min\{u_\lambda(x),u(d)-\ep\},u(c)+\ep\},$$
whose energy is smaller than the original $u_\lambda$, then the assumptions of Proposition~\ref{prop:gatto} are satisfied with $A:=u(c)+\ep$, $B:=u(d)-\ep$, some integer $k_0$ depending on $\ep$ and $u$, and for every $\lambda$ sufficiently small, so passing to the limit $\lambda\to 0^+$ in (\ref{th:gatto}) and then letting $\ep \to 0^+$, we get (\ref{liminf_leb_pts}).

At this point, we can take more Lebesgue points $x_1<\dots <x_m$ for $u$ in $(a,b)$ and applying the (\ref{liminf_leb_pts}) in each interval $(x_i,x_{i+1})$ we deduce that
\begin{align*}
\liminf_{\lambda\to +\infty} \Fgpl(u_{\lambda},(a,b))&\geq
\liminf_{\lambda\to +\infty} \sum_{i=1} ^{m-1} \Fgpl(u_{\lambda},(x_{i},x_{i+1})) \\
&\geq
2^{1+\gamma}\log 2 \sum_{i=1} ^{m-1} \frac{(u(x_{i+1})-u(x_{i}))^p}{(x_{i+1}-x_{i})^{p-1}}.
\end{align*}

This concludes the proof of (\ref{th:1D}), because the supremum of the sum in the right-hand side over all the possible choices of an integer $m\geq 2$ and of $m$ Lebesgue points $x_1<\dots <x_m$ is exactly $F_p(u,(a,b))$.

The case of a general open set $\Omega\subseteq \R$ is then proved by approximation with finite unions of intervals contained in $\Omega$, exactly as in \cite{GP-gammaliminf}.

Finally, the result can be extended to the general case $\Omega\subseteq \R^N$ with $N>1$, exploiting the usual slicing technique used, for example, in \cite{2020-APDE-AGMP,2023-Picenni,2023-Lahti,GP-gammaliminf}, without any relevant modification.

However, since we will also need this in the next section, let us introduce some notation and recall the main ideas.

Let $\sigma\in \s^{N-1}$ be a unit vector, and let $\sigma^\perp$ be the subspace of $\R^N$ orthogonal to $\sigma$. Let $\Omega_\sigma \subseteq \sigma^\perp$ denote the orthogonal projection of $\Omega$ onto $\sigma^\perp$.

For every $z\in \Omega_\sigma$ we set $\Omega_{\sigma,z}:=\{t\in\R: z+\sigma t \in \Omega \}$, and we consider the one-dimensional function $u_{\sigma,z}:\Omega_{\sigma,z} \to \R$ defined by $u_{\sigma,z}(t):=u(z+\sigma t)$.

Then it turns out that both the functionals (\ref{defn:Fgpl}) and (\ref{defn:Fp}) can be computed as an average of their one-dimensional versions computed over all one-dimensional slices $u_{\sigma,z}$, namely it holds that (see \cite{AFP})
\begin{equation}\label{sectioning_norm}
\int_{\s^{N-1}} d\mathcal{H}^{N-1}(\sigma) 
\int_{\Omega_\sigma} F_p(u_{\sigma,z}, \Omega_{\sigma,z}) \,dz = 
C_{N,p}\cdot F_p (u,\Omega),
\end{equation}
and (see \cite{2023-Picenni,GP-gammaliminf} or the general formula in \cite[page~622]{2020-APDE-AGMP})
\begin{equation}\label{sectioning_functionals}
\Fgpl(u,\Omega)=
\frac{1}{2}\int_{\s^{N-1}} d\mathcal{H}^{N-1}(\sigma) 
\int_{\Omega_\sigma} \Fgpl(u_{\sigma,z},\Omega_{\sigma,z})\,dz,
\end{equation}
for every $u\in L^1_\loc(\Omega)$ and every $\gamma\in\R$, $p\geq 1$, $\lambda>0$.

Then (\ref{th:liminf_ineq}) follows from the previous identities and Fatou lemma, since
\begin{align*}
\liminf_{\lambda\to 0^+}\Fgpl(u_\lambda,\Omega)&\geq
\frac{1}{2}\int_{\s^{N-1}} d\mathcal{H}^{N-1}(\sigma) 
\int_{\Omega_\sigma} \liminf_{\lambda\to 0^+} \Fgpl(u_{\sigma,z},\Omega_{\sigma,z})\,dz\\
&\geq 2^{\gamma} \log 2 \int_{\s^{N-1}} d\mathcal{H}^{N-1}(\sigma) 
\int_{\Omega_\sigma} F_p(u_{\sigma,z}, \Omega_{\sigma,z})\,dz\\
&=C_{N,p}\cdot 2^\gamma\log 2 \cdot F_p (u,\Omega).
\end{align*}
\end{proof}

\begin{rmk}\label{rmk:convergence}
\begin{em}
We point out that the only properties of $L^1_{\loc}(\Omega)$ convergence that we need is that it can be transferred to almost every one-dimensional slice (at least on a subsequence) and that in the one-dimensional case it implies that the assumption (\ref{hp:density_ep}) holds for every $u_{\lambda,\ep}$ with $\lambda$ small enough, whenever the endpoints are Lebesgue points for $u$. As a consequence, Theorem~\ref{thm:gamma-liminf} holds also if we replace $L^1_{\loc}$ convergence with weaker notions like convergence in measure or pointwise (almost everywhere) convergence.
\end{em}
\end{rmk}

\section{Poincaré inequality}\label{sec:Poincaré}

This section is devoted to the study of the Poincaré inequality for the functionals (\ref{defn:Fgpl}) in the two regimes $\gamma<-p$ and $\gamma>0$.

\subsection{The case $\gamma<-p$: proof of Theorem~\ref{thm:poincaré_<-p}}

In this subsection we prove Theorem~\ref{thm:poincaré_<-p}. The argument relies again on the dyadic representation and the lower bound for the dyadic functional. However, in order to apply the dyadic estimate, we need to ensure that the assumptions are verified for sufficiently many dyadic intervals. To this end, we need the following lemma.

\begin{lemma}\label{lemma:meas_theory}
Let $\Omega\subset \R^N$ be a measurable set with finite measure, and let $p\geq 1$ be a real number. Let also $u\in L^p(\Omega)$ be a function and let us set
\begin{equation}\label{defn:phi_u}
\phi_u(x,y):=\begin{cases}
1&\text{\emph{if} }|u(y)-u(x)|^p > \frac{1}{4\Leb^N(\Omega)^2}\iint_{\Omega^2}|u(y')-u(x')|^p\,dx'dy',\\
0&\text{\emph{if} }|u(y)-u(x)|^p\leq \frac{1}{4\Leb^N(\Omega)^2}\iint_{\Omega^2}|u(y')-u(x')|^p\,dx'dy',
\end{cases}
\end{equation}

Then there exists a constant $r_0>0$, depending on $p$ and on $\Leb^N(\Omega)$, such that
$$\iint_{\Omega^2}|u(y)-u(x)|^p\,dxdy \leq 4 \iint_{\Omega^2} |u(y)-u(x)|^p \mathbbm{1}_{\{|y-x|>r_0\}}(x,y)\phi_u(x,y)\,dxdy.$$
\end{lemma}

\begin{proof}
We observe that for every $z\in\Omega$ it holds that
$$|u(y)-u(x)|^p\leq 2^{p-1}\left(|u(y)-u(z)|^p+|u(x)-u(z)|^p\right),$$
and therefore, integrating over $\Omega$ with respect to the variable $z$, we get
$$|u(y)-u(x)|^p\leq \frac{2^{p-1}}{\Leb^N(\Omega)}\int_\Omega\left(|u(y)-u(z)|^p+|u(x)-u(z)|^p\right)dz,$$
and then, fixing $r>0$, and integrating over the set where $|y-x|<r$ with respect to the $(x,y)$ variables, we obtain that
\begin{align*}
\lefteqn{\iint_{\Omega^2}|u(y)-u(x)|^p\mathbbm{1}_{\{|y-x|<r\}}(x,y)\,dxdy} \\
&\qquad\leq \frac{2^{p-1}}{\Leb^N(\Omega)} \iint_{\Omega^2} \mathbbm{1}_{\{|y-x|<r\}}(x,y)\,dxdy \int_\Omega\left(|u(y)-u(z)|^p+|u(x)-u(z)|^p\right)dz,\\
&\qquad = 2\cdot \frac{2^{p-1}}{\Leb^N(\Omega)}\int_{\Omega}dx \int_\Omega |u(x)-u(z)|^p\,dz \int_{\Omega}\mathbbm{1}_{\{|y-x|<r\}}(x,y)\,dy,\\
&\qquad\leq \frac{2^{p}\Leb^N(B_r)}{\Leb^N(\Omega)}\iint_{\Omega^2}|u(x)-u(z)|^p\,dxdz.
\end{align*}

Therefore, if we fix $r_0>0$ in such a way that the measure of the $N$-dimensional ball with radius equal to $r_0$ satisfies
\begin{equation}\label{defn:r_0}
\Leb^N(B_{r_0})\leq \frac{\Leb^N(\Omega)}{2^{p+1}},
\end{equation}
then we obtain that
$$\iint_{\Omega^2}|u(y)-u(x)|^p\mathbbm{1}_{\{|y-x|<r_0\}}(x,y)\,dxdy\leq \frac{1}{2}\iint_{\Omega^2}|u(y)-u(x)|^p\,dxdy,$$
and hence we deduce that
\begin{equation}\label{est:lemma_meas_1}
\iint_{\Omega^2}|u(y)-u(x)|^p\mathbbm{1}_{\{|y-x|>r_0\}}(x,y)\,dxdy\geq \frac{1}{2}\iint_{\Omega^2}|u(y)-u(x)|^p\,dxdy.
\end{equation}

Finally, we observe that
\begin{align}
\lefteqn{\iint_{\Omega^2}|u(y)-u(x)|^p\mathbbm{1}_{\{|y-x|>r_0\}}(x,y)(1-\phi_u(x,y))\,dxdy} \nonumber\\
&\qquad\qquad\leq \iint_{\Omega^2}\left(\frac{1}{4\Leb^N(\Omega)^2}\iint_{\Omega^2}|u(y')-u(x')|^p\,dy'dx'\right) dxdy \nonumber\\
&\qquad\qquad\leq \frac{1}{4}\iint_{\Omega^2}|u(y)-u(x)|^p\,dxdy.\label{est:lemma_meas_2}
\end{align}
and hence, combining (\ref{est:lemma_meas_1}) and (\ref{est:lemma_meas_2}), we deduce that
$$\iint_{\Omega^2}|u(y)-u(x)|^p\mathbbm{1}_{\{|y-x|>r_0\}}(x,y)\phi_u(x,y)\,dxdy\geq \frac{1}{4}\iint_{\Omega^2}|u(y)-u(x)|^p\,dxdy,$$
as required.
\end{proof}

We can now proceed with the proof of Theorem~\ref{thm:poincaré_<-p}, starting again with the case $N=1$, in which we prove the following.

\begin{thm}\label{thm:poincaré_<-p_1d}
Let $(a,b)\subset \R$ be an interval, and let $p\geq 1$ and $\gamma<-p$ be real numbers. Then there exists a constant $K>0$, depending only on $\gamma$ and $p$, such that
\begin{equation}\label{poincaré_1d_<-p}
\iint_{(a,b)^2}|u(y)-u(x)|^p\,dxdy
\leq K\cdot\left((b-a)^{p+1} \Fgpl(u,(a,b)) + \lambda^p (b-a)^{2+p+\gamma}\right),
\end{equation}
for every $u\in L^p((a,b))$ and every $\lambda>0$.
\end{thm}

\begin{proof}
Let us consider first the case $(a,b)=(0,1)$, and let $r_0>0$ be the constant given by Lemma~\ref{lemma:meas_theory} with $\Omega=(0,1)$. We observe that if
$$\lambda\geq \left(\frac{r_0}{2}\right)^{-1-\gamma/p} \left(\frac{1}{4}\iint_{(0,1)^2}|u(y)-u(x)|^p\,dxdy\right)^{1/p},$$
then (\ref{poincaré_1d_<-p}) holds trivially with $K=4\cdot (r_0/2)^{p+\gamma}$, so we can assume that
\begin{equation}\label{hp:lambda}
\lambda< \left(\frac{r_0}{2}\right)^{-1-\gamma/p}\left(\frac{1}{4}\iint_{(0,1)^2}|u(y)-u(x)|^p\,dxdy\right)^{1/p}.
\end{equation}

In this case, for every $t>0$, let us consider the functions
\begin{equation}\label{defn:k_delta}
k(t):=\lfloor\log_2 (1/t)\rfloor, \qquad \text{and} \qquad \delta(t):= 2^{k(t)}t,
\end{equation}
which are defined so that $\delta(t)$ is the unique number in $(1/2,1]$ that can be written as $2^k t$ for some $k\in\Z$.

For every $x\in \R$ and $\delta>0$, let us consider also the function $v_{x,\delta}(z):=u(x+\delta z)$, as in (\ref{defn:v_x,delta}). Then, for every $0<x<y<1$, it holds that
$$v_{x,\delta(y-x)}(0)=u(x) \qquad v_{x,\delta(y-x)}(2^{-k(y-x)})=u(y).$$

Moreover, if
$$|u(y)-u(x)|>\left(\frac{1}{4}\iint_{(0,1)^2}|u(y')-u(x')|^p\,dx'dy'\right)^{1/p}\qquad \text{and}\qquad y-x>r_0,$$
then from (\ref{hp:lambda}) it follows that (recall that $-1-\gamma/p>0$)
$$\lambda< \left(\frac{y-x}{2}\right)^{-1-\gamma/p}|u(y)-u(x)|\leq 2^{(k(y-x)+1)(1+\gamma/p)} \bigl|v_{x,\delta(y-x)}(2^{-k(y-x)}) - v_{x,\delta(y-x)}(0)\bigr|,$$
where the second inequality follows from the inequality $y-x\leq 2^{-k(y-x)}$.

This is exactly what we need to apply Proposition~\ref{prop:est_DF} in the interval $[0,2^{-k(y-x)}]$, hence for every $(x,y)\in (0,1)^2$ we deduce that
\begin{multline*}
|u(y)-u(x)|^p \mathbbm{1}_{\{y-x>r_0,\ x<y\}}(x,y)\phi_u(x,y)\\
\leq 2^{-\gamma}\lambda^p \delta(y-x)^{p+\gamma}2^{-k(y-x)(p-1)} \DFgpl\bigl(\delta(y-x),v_{x,\delta(y-x)},[0,2^{-k(y-x)}]\bigr),
\end{multline*}
where $\phi_u$ is defined as in (\ref{defn:phi_u}).

Integrating over $(0,1)^2$ and applying Lemma~\ref{lemma:meas_theory}, we get
\begin{align*}
\iint_{(0,1)^2}|u(y)-u(x)|^p \,dxdy& \leq 4\iint_{(0,1)^2}|u(y)-u(x)|^p \mathbbm{1}_{\{|y-x|>r_0\}}(x,y)\phi_u(x,y)\,dxdy\\
& = 8 \iint_{(0,1)^2}|u(y)-u(x)|^p \mathbbm{1}_{\{|y-x|>r_0,\ x<y\}}(x,y)\phi_u(x,y)\,dxdy\\
&\leq 2^{3-\gamma}\lambda^p \int_{0} ^{1-r_0}dx\int_{x+r_0} ^{1} 2^{-k(y-x)(p-1)} \delta(|y-x|)^{p+\gamma} \\
&\hspace{5em} \cdot\DFgpl\bigl(\delta(|y-x|),v_{x,\delta(|y-x|)},[0,2^{-k(|y-x|)}]\bigr)\,dy.
\end{align*}

Using a change of variables, some simple estimates and the identity $2^{-k(t)}\delta(t)=t$, we obtain that
\begin{align}
\lefteqn{\iint_{(0,1)^2}|u(y)-u(x)|^p \,dxdy}\nonumber\\
&\quad\leq 2^{3-\gamma}\lambda^p \int_{0} ^{1-r_0}dx \int_{r_0} ^{1-x} 2^{-k(t)(p-1)}\delta(t)^{p+\gamma} \DFgpl\bigl(\delta(t),v_{x,\delta(t)},[0,2^{-k(t)}]\bigr)\,dt\nonumber\\
&\quad\leq 2^{3-\gamma}\lambda^p \int_{0} ^{1-r_0}dx \int_{r_0} ^{1} t^{p-1} \delta(t)^{\gamma+1} \DFgpl\bigl(\delta(t),v_{x,\delta(t)},(-x/\delta(t),(1-x)/\delta(t))\bigr)\,dt.\label{est:Lp-DF}
\end{align}

Now we observe that $\delta(2^{-k}t)=\delta(t)$ for every $t\in (0,1)$ and every $k\in\N$ so it holds that
\begin{align*}
\lefteqn{ \int_{r_0} ^{1} t^{p-1}\delta(t)^{\gamma+1} \DFgpl\bigl(\delta(t),v_{x,\delta(t)},(-x/\delta(t),(1-x)/\delta(t))\bigr)\,dt}\\
&\qquad\qquad\leq \sum_{k=0} ^{k(r_0)} \int_{2^{-k-1}} ^{2^{-k}} t^{p-1}\delta(t)^{\gamma+1} \DFgpl\bigl(\delta(t),v_{x,\delta(t)},(-x/\delta(t),(1-x)/\delta(t))\bigl)\,dt\\
&\qquad\qquad= \sum_{k=0} ^{k(r_0)} \int_{1/2} ^{1} t^{p-1}\delta(t)^{\gamma+1} \DFgpl\bigl(\delta(t),v_{x,\delta(t)},(-x/\delta(t),(1-x)/\delta(t))\bigr)2^{-kp}\,dt
\end{align*}

Plugging this estimate into (\ref{est:Lp-DF}) and observing that
$$\sum_{k=0} ^{k(r_0)} 2^{-kp}\leq \sum_{k=0} ^{\infty} 2^{-kp}\leq 2,$$
and that $\delta(t)=t$ for every $t\in (1/2,1)$, we obtain that
\begin{align*}
\lefteqn{\iint_{(0,1)^2}|u(y)-u(x)|^p \,dxdy}\\
&\qquad\qquad\leq 2^{4-\gamma}\lambda^p \int_{0} ^{1-r_0}dx \int_{1/2} ^{1} t^{p+\gamma} \DFgpl\bigl(t,v_{x,t},(-x/t,(1-x)/t)\bigr)\,dt\\
&\qquad\qquad\leq 2^{4-\gamma}\lambda^p \int_{0} ^{1}dx \int_{1/2} ^{1} t^{\gamma-1} \DFgpl\bigl(t,v_{x,t},(-x/t,(1-x)/t)\bigr)\,dt.
\end{align*}

Since the function
$$x\mapsto \DFgpl\bigl(t,v_{x,t},(-x/t,(1-x)/t)\bigr)$$
is $t$-periodic, we deduce that
\begin{multline*}
\iint_{(0,1)^2}|u(y)-u(x)|^p \,dxdy\\
\leq 2^{4-\gamma}\cdot 2\lambda^p\int_{1/2} ^{1} t^{\gamma-1}\,dt \int_{0} ^t \DFgpl\bigl(t,v_{x,t},(-x/t,(1-x)/t)\bigr)\,dx,
\end{multline*}
so Proposition~\ref{prop:dyadic_repr} yields 
$$\iint_{(0,1)^2}|u(y)-u(x)|^p \,dxdy \leq 2^{4-\gamma}\Fgpl(u,(0,1)).$$

This means that (\ref{poincaré_1d_<-p}) holds with $K=2^{4-\gamma}$ under the assumption (\ref{hp:lambda}), so we conclude that (\ref{poincaré_1d_<-p}) holds for every $\lambda>0$ with
\begin{equation}\label{defn:K_1d}
K=\max\{4\cdot (r_0/2)^{p+\gamma},2^{4-\gamma}\}.
\end{equation}

Finally, we observe that the case of a general interval $(a,b)\subset \R$ follows from the case $(a,b)=(0,1)$ just by scaling. 

Indeed, given a function $u\in L^p((a,b))$, if we set
$$w(x)=u(a+(b-a)x), \qquad \lambda'= (b-a)^{1+\gamma/p}\cdot \lambda,$$
then it holds that
$$\iint_{(a,b)^2}|u(y)-u(x)|^p\,dxdy=(b-a)^2 \iint_{(0,1)^2} |w(y')-w(x')|^p\,dx'dy',$$
and
$$\Fgpl(u,(a,b))= (b-a)^{1-p} F_{\gamma,p,\lambda'}(w,(0,1)),$$
so (\ref{poincaré_1d_<-p}) follows from the corresponding inequality applied to $w$ in $(0,1)$, with the same constant $K$ defined in (\ref{defn:K_1d}), where $r_0$ still corresponds to the domain $(0,1)$.
\end{proof}

At this point, we are ready to prove Theorem~\ref{thm:poincaré_<-p} in full generality. The proof relies on the one-dimensional result and the usual slicing method, even though the application here is not as straightforward as in the proof of the Gamma-liminf estimate, and requires again (a part of) Lemma~\ref{lemma:meas_theory}.

\begin{proof}[Proof of Theorem~\ref{thm:poincaré_<-p}]
Since the case $N=1$ follows from Theorem~\ref{thm:poincaré_<-p_1d}, let us assume that $N\geq 2$.

From Lemma~\ref{lemma:meas_theory}, and actually from (\ref{est:lemma_meas_1}), we deduce that
$$\iint_{\Omega^2}|u(y)-u(x)|^p\,dxdy\leq 2\iint_{\Omega^2}|u(y)-u(x)|^p\mathbbm{1}_{\{|y-x|>r_0\}}(x,y)\,dxdy.$$

Then, using the notation introduced at the end of Section~\ref{sec:gamma-liminf}, we set $y=x+\sigma t$, then $x=z+\sigma s$ and we apply Fubini theorem (equivalently we can directly use the slicing formula from \cite[page~622]{2020-APDE-AGMP}), so we can write
\begin{multline*}
\lefteqn{\iint_{\Omega^2}|u(y)-u(x)|^p\mathbbm{1}_{\{|y-x|>r_0\}}(x,y)\,dxdy}\\
=\int_{\s^{N-1}} d\mathcal{H}^{N-1}(\sigma) 
\int_{\Omega_\sigma}dz \iint_{\Omega_{\sigma,z}^2} |u_{\sigma,z}(t)-u_{\sigma,z}(s)|^p \mathbbm{1}_{\{|t-s|>r_0\}}(s,t)|t-s|^{N-1}\,dtds.
\end{multline*}

We point out that the convexity of $\Omega$ ensures that the sets $\Omega_{\sigma,z}$ are intervals, and in particular the last integral vanishes if $\Leb^1(\Omega_{\sigma,z})\leq r_0$, so we deduce that
\begin{multline*}
\int_{\s^{N-1}} d\mathcal{H}^{N-1}(\sigma) 
\int_{\Omega_\sigma}dz \iint_{\Omega_{\sigma,z}^2} |u_{\sigma,z}(t)-u_{\sigma,z}(s)|^p \mathbbm{1}_{\{|t-s|>r_0\}}(s,t)|t-s|^{N-1}\,dtds\\
\leq \diam(\Omega)^{N-1} \int_{\s^{N-1}} d\mathcal{H}^{N-1}(\sigma) 
\int_{\Omega_\sigma} dz\, \chi_{r_0}(\Omega_{\sigma,z}) \iint_{\Omega_{\sigma,z}^2} |u_{\sigma,z}(t)-u_{\sigma,z}(s)|^p \,dtds,
\end{multline*}
where
$$\chi_{r_0}(\Omega_{\sigma,z}):=\begin{cases}
1&\text{if }\Leb^1(\Omega_{\sigma,z})>r_0,\\
0&\text{if }\Leb^1(\Omega_{\sigma,z})\leq r_0.
\end{cases}$$

Now we apply Theorem~\ref{thm:poincaré_<-p_1d}, so we obtain that
$$\iint_{\Omega_{\sigma,z}^2} |u_{\sigma,z}(t)-u_{\sigma,z}(s)|^p \,dtds \leq K\cdot\left(\Leb^1(\Omega_{\sigma,z})^{p+1} \Fgpl(u_{\sigma,z},\Omega_{\sigma,z}) + \lambda^p \Leb^1(\Omega_{\sigma,z})^{2+p+\gamma}\right)$$

We observe that $r_0\leq \Leb^1(\Omega_{\sigma,z})\leq \diam(\Omega)$, for every $\sigma \in \s^{N-1}$ and every $z\in \Omega_\sigma$ for which $\chi_{r_0}(\Omega_{\sigma,z})\neq 0$, so we can absorb the powers of $\Leb^1(\Omega_{\sigma,z})$ into the constant $K$, independently of the sign of the exponents.

Thus, if we allow the constant $K$ to change value from line to line (but still depending only on $\Omega,N,p,\gamma$), combining the previous estimates, we find that
$$\iint_{\Omega^2}|u(y)-u(x)|^p\,dxdy\leq K \int_{\s^{N-1}} d\mathcal{H}^{N-1}(\sigma) 
\int_{\Omega_\sigma} \left(\Fgpl(u_{\sigma,z},\Omega_{\sigma,z})+\lambda^p\right) dz,$$
and finally, recalling (\ref{sectioning_functionals}), we obtain (\ref{th:poincaré_<-p})
\end{proof}

\begin{rmk}[Estimates for $K$]
\begin{em}
We point out that an explicit value for the constant $K$ in Theorem~\ref{thm:poincaré_<-p_1d} can be easily computed, since the number $r_0$ in (\ref{defn:K_1d}) comes from (\ref{defn:r_0}) with $N=1$ and $\Omega=(0,1)$. It would be also possible to optimize the choice of $r_0$ in Lemma~\ref{lemma:meas_theory} to find a better estimate for $K$.

In the higher dimensional case, following our proof one could find an explicit expression for $K$ in terms of $\Leb^N(\Omega)$ and $\diam (\Omega)$.  However, if $\Omega$ has a simple geometry, for example if $\Omega$ is a ball or a cube, then the length of $\Omega_{\sigma,z}$ can be computed explicitly, so one could find a more precise value.
\end{em}
\end{rmk}

\subsection{The case $\gamma>0$: counterexample and proof of Theorem~\ref{thm:poincaré_>0}}

In this subsection we investigate the validity of the Poincaré inequality when $\gamma>0$. We start with the following simple counterexample, which provides a negative answer to \cite[Open question~33]{Nguyen-survey} in the case $\gamma>0$.

\begin{prop}\label{prop:counterex_>0}
There exists a sequence of functions $u_j:(-1,1)\to \R$ such that
$$\lim_{j\to +\infty}\iint_{(-1,1)^2}|u_j(y)-u_j(x)|^p=+\infty \qquad \text{and}\qquad \limsup_{j\to +\infty}\Fgpl(u_j,(-1,1))<+\infty,$$
for every $\gamma>0$, $p\geq 1$ and $\lambda>0$.
\end{prop}

\begin{proof}
For every $j\in\N$ we set
$$u_j(x):=\begin{cases}
    0 &\text{if }x\leq 0,\\
    j &\text{if }x>0.
\end{cases}$$

Then it holds that
$$\iint_{(-1,1)^2}|u_j(y)-u_j(x)|^p = 2j^p,$$
and
\begin{align*}
\Fgpl(u_j,(-1,1))&\leq 2\lambda^p \nu_\gamma\bigl((-1,0)\times (0,1)\bigr)\\
&=2\lambda^p\int_{-1} ^0 dx \int_0 ^1 (y-x)^{\gamma-1} \,dy\\
&=\frac{4\lambda^p (2^{\gamma}-1)}{\gamma(\gamma+1)},
\end{align*}
so the thesis follows.
\end{proof}

We can now prove Theorem~\ref{thm:poincaré_>0} in the one-dimensional case, with explicit dependence on the domain.

\begin{thm}\label{thm:poincaré_1d_>0}
Let $(a,b)\subset \R$ be an interval, and let $p\geq 1$ and $\gamma>0$ be real numbers. Then there exists a constant $K>0$, depending only on $\gamma$ and $p$, such that
\begin{multline}\label{poincaré_1d_>0}
\iint_{(a,b)^2}|u(y)-u(x)|^p\,dxdy\\
\leq K\cdot\left((b-a)^{p+1} \Fgpl(u,(a,b)) + \lambda^{-p} (b-a)^{2-p-\gamma}\osc(u,(a,b))^{2p} \right),
\end{multline}
for every $u\in L^p(\Omega)$ and every $\lambda>0$.
\end{thm}
\begin{proof}
Arguing by scaling as in the proof of Theorem~\ref{thm:poincaré_<-p_1d} we can reduce ourselves to the case $(a,b)=(0,1)$.

In this case, we observe that if
$$\lambda\leq \left(\frac{2}{r_0}\right)^{1+\gamma/p}\osc(u,(0,1)),$$
then (\ref{poincaré_1d_>0}) holds trivially with $K=(2/r_0)^{p+\gamma}$, since of course
$$\iint_{(0,1)^2}|u(y)-u(x)|^p\,dxdy\leq \osc(u,(0,1)).$$

Hence we can assume that
$$\lambda> \left(\frac{2}{r_0}\right)^{1+\gamma/p}\osc(u,(0,1)).$$

Then, using the notation introduced in (\ref{defn:v_x,delta}) and (\ref{defn:k_delta}), for every $0<x<y<1$ with $y-x>r_0$ we deduce that
\begin{equation}\label{hp:est_GP}
\lambda>\left(\frac{2}{y-x}\right)^{1+\gamma/p}\osc(u,(0,1))\geq 2^{(k(y-x)+1)(1+\gamma/p)}\bigl|v_{x,\delta(y-x)}(2^{-k(y-x)}) - v_{x,\delta(y-x)}(0)\bigr|,
\end{equation}

Thus, we can apply \cite[Proposition~2.2]{GP-gammaliminf} which implies that for every $(x,y)\in (0,1)^2$ it holds that
\begin{multline*}
|u(y)-u(x)|^p \mathbbm{1}_{\{y-x>r_0,\ x<y\}}(x,y)\\
\leq (2^{\gamma+1}-1) \lambda^p \delta(y-x)^{p+\gamma}2^{-k(y-x)(p-1)} \DFgpl\bigl(\delta(y-x),v_{x,\delta(y-x)},[0,2^{-k(y-x)}]\bigr).
\end{multline*}

We remark that \cite[Proposition~2.2]{GP-gammaliminf} was written for domains $[\alpha,\beta]$ with integer endpoints, under the assumption that $\lambda>2^{1+\gamma/p}|v(\beta)-v(\alpha)|$. However, by scaling (see \cite[formula~(3.5)]{GP-gammaliminf}), the same results applies when the endpoints of the domain belong to $2^{-k}\Z$, under the assumption that $\lambda>2^{(k+1)(1+\gamma/p)}|v(\beta)-v(\alpha)|$, and in our case this is exactly (\ref{hp:est_GP}).

From this point we can conclude exactly as in the proof of Theorem~\ref{thm:poincaré_<-p_1d}.
\end{proof}

\begin{proof}[Proof of Theorem~\ref{thm:poincaré_>0}]
Given Theorem~\ref{thm:poincaré_1d_>0}, then Theorem~\ref{thm:poincaré_>0} follows from the same slicing argument used to pass from Theorem~\ref{thm:poincaré_<-p_1d} to Theorem~\ref{thm:poincaré_<-p}.
\end{proof}

\section{Open problems}\label{sec:op}
In this section we describe some problems that remains open.

The first one concerns the extension of the Gamma-liminf estimates to the range $\gamma\in (-p,-1]$ or, equivalently, improving (\ref{liminf_est_nguyen}) by replacing $u$ with a sequence $u_\lambda\to u$.
\begin{open}
Let $p\geq 1$ and $\gamma\in (-p,-1]$. Do there exist a positive constant $c(N,\gamma,p)>0$ such that
$$\Gamma-\liminf_{\lambda\to 0^+} \Fgpl(u,\R^N)\geq c(N,\gamma,p)\cdot F_p(u,\R^N)\qquad \forall u \in L^1_{\loc} (\R^N)?$$
\end{open}

We recall that such estimate is known to fail for $\gamma\in (-1,0)$ (see \cite[Proposition~17 and Theorem~29]{Nguyen-survey}), while it is proved in \cite{GP-gammaliminf} for $\gamma>0$, in \cite{2008-JEMS-Nguyen} (with roots in \cite{2006-CRAS-BN}) for $\gamma=-p$ (see also \cite{2020-APDE-AGMP} for the sharp version), and in Theorem~\ref{thm:gamma-liminf} for $\gamma<-p$. 

A related problem is the validity of the Poincaré inequality when $\gamma\in(-p,-1]$, which was posed \cite[Open question~33]{Nguyen-survey} and answered in the present paper for $\gamma\in \R\setminus [-p,0]$. We recall that the case $\gamma=-p$ was established in \cite{2011-CalcVar-Nguyen}.

\begin{open}
Let $p\geq 1$ and $\gamma\in (-p,-1]$ and let $\Omega\subset \R^N$ be a convex set (or even a ball). Do there exists a constant $K(\Omega,N,\gamma,p)>0$ such that
$$\iint_{\Omega \times \Omega} |u(y)-u(x)|^p\,dxdy \leq K(\Omega,N,\gamma,p)\cdot \left(\Fgpl(u,\Omega) + \lambda ^p \right),$$
for every $u\in L^p(\Omega)$ and every $\lambda>0$?
\end{open}

Finally, we believe that the most relevant problem toward the understanding of the asymptotic behavior of the functionals (\ref{defn:Fgpl}) is whether the gap between the pointwise limit and the Gamma-limit is a peculiarity of the case $\gamma=-p$, or also occurs for other values of $\gamma$. Thus, we pose the following questions.

\begin{open}
Let $p>1$ and $\gamma>0$. Is it true that
\begin{equation}\label{open_p>1}
\Gamma-\lim_{\lambda\to +\infty} \Fgpl(u,\R^N)=C_{N,p}\cdot \frac{1}{\gamma} \cdot F_p(u,\R^N) \qquad \forall u\in\ L^1_{\loc}(\R^N)?
\end{equation}
\end{open}

\begin{open}\label{op_p=1_gamma>0}
Let $p=1$ and $\gamma>0$. Is it true that
\begin{equation}\label{open_p=1}
\Gamma-\lim_{\lambda\to +\infty} F_{\gamma,1,\lambda}(u,\R^N)=C_{N,1}\cdot \frac{1}{\gamma+1} \cdot F_p(u,\R^N) \qquad \forall u\in\ L^1_{\loc}(\R^N)?
\end{equation}
\end{open}

\begin{open}
Let $p>1$ and $\gamma \in (-\infty,-p)\cup (-p,-1]$. Is it true that
$$\Gamma-\lim_{\lambda\to 0^+} \Fgpl(u,\R^N)=C_{N,p}\cdot \frac{1}{|\gamma|} \cdot F_p(u,\R^N) \qquad \forall u\in\ L^1_{\loc}(\R^N)?$$
\end{open}

In the case $p=1$ and $\gamma \in (-\infty,-p)\cup (-p,-1]$, we do not make a conjecture, since the pointwise limit has not been appropriately investigated and the constant coming from the smooth case is unlikely to be optimal, but the corresponding problem certainly deserves the same attention of the previous ones.

We also remark that, even if (\ref{open_p=1}) turns out to be true, it might also be reasonable that in (\ref{open_p>1}) the factor $1/\gamma$ should be replaced by some constant $c(\gamma,p)>0$ converging to $1/(\gamma+1)$ when $p\to 1^+$, namely there might be a gap only for $p>1$.

Anyway, if any of the problem above had a negative answer, then it would be interesting to identify the correct Gamma-limit. Actually, more than the explicit value of the constants, the relevant question would be to understand the shape of the recovery sequences for smooth functions.

\appendix

\section{Monotonicity in the dyadic setting}

In this appendix we show that there exists a procedure to transform a given function, defined on dyadic numbers, into a monotone function without increasing the value of the dyadic functional. This is not necessary for our analysis, and does not simplify any of the previous proofs, but we include it here in the hope that it might be useful for some future research.

Indeed, a similar statement (although based on a completely different procedure) was crucial in the proof of the sharp liminf inequality for the case $\gamma=-p$ in \cite{2020-APDE-AGMP}, which is the only case in which the exact Gamma-limit is known.

Furthermore, if one could prove a similar result in the continuous setting, namely for the functionals (\ref{defn:Fgpl}) and for functions defined on an interval, then this would immediately imply a positive answer to Open problem~\ref{op_p=1_gamma>0}, as observed in \cite[Remark~3.28]{2023-Lahti}.

Unfortunately, if we start from some $v:(a,b)\to \R$ and we consider different dyadic discretizations of the form $(\ref{defn:v_x,delta})$, corresponding to different choices of $(x,\delta)$, then the monotone functions given by our procedure might not be compatible, namely they might not be discretizations of a single monotone function $\widehat{v}:(a,b)\to \R$. For this reason, we cannot deduce the continuous result from the dyadic one and Proposition~\ref{prop:dyadic_repr}.

\begin{lemma}\label{lemma:monotonicity}
Let $k_0\in\N$ and let $[\alpha,\beta]$ be a $k_0$-dyadic interval and $v:\D\cap [\alpha,\beta]\to \R$ be a function.

Then there exists a monotone function $\widehat{v}:\D\cap [\alpha,\beta]\to \R$ such that
$$\widehat{v}(\alpha) =v(\alpha), \qquad \widehat{v}(\beta) =v(\beta),$$
and
\begin{equation}\label{th:v_hat-v}
\DFgpl(\delta,\widehat{v},[\alpha,\beta]) \leq \DFgpl(\delta,v,[\alpha,\beta]),
\end{equation}
for every $p\geq 1$, $\gamma<-p$, $\lambda>0$ and $\delta \in [1/2,1]$.
\end{lemma}

\begin{proof}
Let us set $A:=v(\alpha)$ and $B:=v(\beta)$. We can assume, without loss of generality, that $A\leq B$.

By truncation, namely by replacing $v(x)$ with $v_0(x):=\max\{\min\{v(x),B\},A\}$, we reduce to the case in which $v_0:[\alpha,\beta]\to [A,B]$, because the truncation does not increase the functional. We point out that it still holds that $v_0(\alpha)=A$ and $v_0(\beta)=B$.

Then we repeat the truncation argument on each of the two halves of the domain, namely we set
$$v_1(x):=\begin{cases}
\max\{\min\{v_0(x),v(\frac{\alpha+\beta}{2})\},A\} &\text{if }x\in \D\cap[\alpha,\frac{\alpha+\beta}{2}],\\
\max\{\min\{v_0(x),B\},v(\frac{\alpha+\beta}{2})\} &\text{if }x\in \D\cap[\frac{\alpha+\beta}{2},\beta].
\end{cases}$$

We observe that $v_1(\tfrac{\alpha+\beta}{2})=v_0(\tfrac{\alpha+\beta}{2})$ and that $v_1$ has the property that $v_1(x)\leq v_1(y)$ for every $x\in \D\cap[\alpha,\frac{\alpha+\beta}{2}]$ and every $y\in \D\cap[\frac{\alpha+\beta}{2},\beta]$. Moreover, it holds that
$$\DFgpl(\delta,v_1,[\alpha,\tfrac{\alpha+\beta}{2}])\leq \DFgpl(\delta,v_0,[\alpha,\tfrac{\alpha+\beta}{2}])\qquad \forall \delta \in [1/2,1],$$
and
$$\DFgpl(\delta,v_1,[\tfrac{\alpha+\beta}{2},\beta])\leq \DFgpl(\delta,v_0,[\tfrac{\alpha+\beta}{2},\beta])\qquad \forall \delta \in [1/2,1],$$
because $v_1$ is a truncation of $v_0$ in each of the two subintervals. As a consequence, since the dyadic functional does not involve any interaction between points $(x,y)\in [\alpha,\frac{\alpha+\beta}{2}] \times [\frac{\alpha+\beta}{2},\beta]$, we have that
\begin{multline*}
\DFgpl(\delta,v_1,[\alpha,\beta])=\DFgpl(\delta,v_1,[\alpha,\tfrac{\alpha+\beta}{2}])+\DFgpl(\delta,v_1,[\tfrac{\alpha+\beta}{2},\beta])\\
\leq \DFgpl(\delta,v_0,[\alpha,\tfrac{\alpha+\beta}{2}])+\DFgpl(\delta,v_0,[\tfrac{\alpha+\beta}{2},\beta])=\DFgpl(\delta,v_0,[\alpha,\beta]).
\end{multline*}

Now let us assume that for some positive integer $j>0$ we have already defined a function $v_j:[\alpha,\beta]\to [A,B]$ with the property that $v_j(x)\leq v_j(y)$ whenever $x<y$ belong to different $(k_0+j)$-dyadic intervals contained in $[\alpha,\beta]$.

Then we inductively define $v_{j+1}$ by setting 
$$v_{j+1}(x):=\max\{\min\{v_j(x),v(\beta'))\},v(\alpha')\} \qquad \forall x\in \D\cap[\alpha',\beta'],$$
for every $(k_0+j+1)$-dyadic interval $[\alpha',\beta']\subset [\alpha,\beta]$.

We observe that $v_{j+1}(x)=v_j(x)$ for every $x=i/2^{k_0+j+1}$ with $i\in \D_{k_0+j+1}([\alpha,\beta])$, and that $v_{j+1}$ has the property that $v_j(x)\leq v_j(y)$ whenever $x<y$ belong to different $(k_0+j+1)$-dyadic intervals contained in $[\alpha,\beta]$. Moreover, since $v_{j+1}$ is obtained by truncation in every $(k_0+j+1)$-dyadic interval, and the dyadic functional does not involve any interaction between points belonging to different $(k_0+j+1)$-dyadic intervals, by the same arguments used above with $v_1$ we get that
\begin{equation}\label{v_j-decrease}
\DFgpl(\delta,v_{j+1},[\alpha,\beta]) \leq \DFgpl(\delta,v_j,[\alpha,\beta]).
\end{equation}

Finally, we set
$$\widehat{v}(x):=\lim_{j\to +\infty} v_j(x)\qquad \forall x\in \D\cap [\alpha,\beta].$$
We observe that the limit in the definition above always exists, because if $x=i/2^k$ with $i\in \D_{k}([\alpha,\beta])$ then the sequence $\{v_j(x)\}$ becomes constant as soon as $k_0+j>k$.

Moreover, the function $\widehat{v}$ is nondecreasing. Indeed, if $x<y$ are dyadic numbers, then there exists an integer $k>0$ such that $x$ and $y$ belong to different $k$-dyadic intervals, and hence $\widehat{v}(x)=v_j(x)\leq v_j(y)=\widehat{v}(y)$ for every $j>k-k_0$.

At this point, we are left to prove (\ref{th:v_hat-v}). To this end we observe that (\ref{small-interval-invisible}) implies that
$$\psigpl\left(\frac{\delta}{2^k}, \left|v_j\left(\frac{i+1}{2^k}\right)-v_j\left(\frac{i}{2^k}\right)\right| \right)=0$$
for every $k$ sufficiently large and every $j\in\N$.

As a consequence, we deduce that for $\gamma<-p$ and for functions with values in $[A,B]$ the first sum in the definition (\ref{defn:DF}) of the dyadic functionals can be truncated at a fixed integer without affecting the value of the functional, and hence
$$\DFgpl(\delta,\widehat{v},[\alpha,\beta]) =\DFgpl(\delta,v_j,[\alpha,\beta]),$$
for every $j$ sufficiently large, so (\ref{th:v_hat-v}) follows from (\ref{v_j-decrease}).
\end{proof}

\subsubsection*{\centering Acknowledgments}

The author is a member of the \selectlanguage{italian} ``Gruppo Nazionale per l'Analisi Matematica, la Probabilità e le loro Applicazioni'' (GNAMPA) of the ``Istituto Nazionale di Alta Matematica'' (INdAM).

\selectlanguage{english}

\subsubsection*{\centering Final statements}

\paragraph{Conflict of interest} The author has no relevant financial or non-financial interests to disclose.



\paragraph{Data availability} No data were used for this research.

\paragraph{AI usage} No AI tool has been used for the research presented in this paper, nor for writing this manuscript.


\end{document}